\documentclass[12pt]{amsart}
\usepackage{amsmath}
\usepackage{amsfonts}
\usepackage{amssymb}
\usepackage{amsthm}
\usepackage{eucal}
\newtheorem{theorem}{Theorem}[section]

\usepackage{mathrsfs}
\usepackage{eucal}
\usepackage{graphicx}
\newcommand{\qq}{\mathbb{Q}}
\newcommand{\zz}{\mathbb{Z}}
\newcommand{\cc}{\mathbb{F}}
\newcommand{\pp}{\mathbb{P}}

\newcommand{\rr}{\mathbb{R}}

\renewcommand{\P}{\mathcal{P}}

\newcommand{\E}{\mathcal{E}}

\renewcommand{\L}{\mathcal{L}}
\newcommand{\M}{\mathcal{M}}
\newcommand{\X}{\mathcal{X}}

\renewcommand{\H}{\mathcal{H}}

\renewcommand{\O}{\mathscr{O}}

\renewcommand{\tilde}{\widetilde}
\usepackage{dutchcal}
\newcommand{\f}{\mathcal{f}}
\newcommand{\B}{\mathcal{B}}

\DeclareMathOperator{\codim}{codim}
\newtheorem{Lemma}[theorem]{Lemma}
\newtheorem{Corollary}[theorem]{Corollary}

\DeclareMathOperator{\Pic}{Pic}
\DeclareMathOperator{\Supp}{Supp}
\DeclareMathOperator{\Jac}{Jac}
\DeclareMathOperator{\End}{End}
\DeclareMathOperator{\Def}{Def}
\DeclareMathOperator{\tr}{Tr}
\DeclareMathOperator{\rank}{rank}
\DeclareMathOperator{\res}{res}
\DeclareMathOperator{\ord}{ord}
\DeclareMathOperator{\diag}{\mathfrak{d}}
\newcommand{\ee}{{\vec{e}} \nobreak\hspace{.16667em plus .08333em}'}
\newcommand{\Sigmabar}{\overline{\Sigma}}
\newcommand{\e}{\vec{e}}
\usepackage{stmaryrd}
\renewcommand{\ll}{\llbracket}
\renewcommand{\rr}{\rrbracket}
\theoremstyle{definition}
\newtheorem{Example}[theorem]{Example}
\newtheorem{Definition}[theorem]{Definition}
\theoremstyle{remark}
\newtheorem*{remark}{Remark}

\theoremstyle{remark}
\newtheorem*{warning}{Warning}
\usepackage{tikz}
\usepackage{tikz-cd}
\usepackage{wasysym, stackengine, makebox, graphicx}

\tikzset{cong/.style={draw=none,edge node={node [sloped, allow upside down, auto=false]{$\cong$}}},
         Isom/.style={draw=none,every to/.append style={edge node={node [sloped, allow upside down, auto=false]{$\cong$}}}}}

\usetikzlibrary{matrix,arrows,decorations.pathmorphing}
\usepackage[margin=1in]{geometry}

\numberwithin{equation}{section}

\title{A refined Brill-Noether theory over Hurwitz spaces}
\author{Hannah K. Larson}
\date{\today}
\begin{document}
\maketitle

\begin{abstract}
Let $f\colon  C \rightarrow \pp^1$ be a degree $k$ genus $g$ cover.
The stratification of line bundles $L \in \Pic^d(C)$ by the splitting type of $f_*L$ is a refinement of the stratification by Brill-Noether loci $W^r_d(C)$.
We prove that for general degree $k$ covers, these strata are smooth of the expected dimension.
In particular, this determines the dimensions of all irreducible components of $W^r_d(C)$ for a general $k$-gonal curve (there are often components of different dimensions), extending results of Pflueger \cite{Pf} and Jensen-Ranganathan \cite{JR}. The results here apply over any algebraically closed field.
\end{abstract}

\section{Introduction}

Brill-Noether theory characterizes the maps of general curves to projective space. 
Degree $d$ maps of a curve $C \rightarrow \pp^r$ correspond to line bundles in the \textit{Brill-Noether locus}
\[W^r_d(C) = \{ L \in \Pic^d(C): h^0(C, L) \geq r+1\}.\] 
The fundamental Brill-Noether-Petri theorem \cite{GH,G} states that for a general curve $C$ of genus $g$, 
\[\dim W^r_d(C) = \rho(g, r, d) := g - (r+1)(g - d + r),\]
and $W^r_d(C)$ is smooth away from $W^{r+1}_d(C)$. Furthermore, if $\rho(g, r, d) > 0$, then $W^r_d(C)$ is irreducible \cite{FL}.

For certain special curves, $W^r_d(C)$ can be reducible and have components of larger than the expected dimension.
In particular, Coppens--Martens showed that for a general $k$-gonal curve, $W^r_d(C)$ has a component of dimension $\rho(g, \alpha - 1, d) - (r+1 - \alpha)k$ for $\alpha = 1, r, r+1$ \cite{CM1}, which they later extended to $\alpha$ dividing $r$ or $r+1$ \cite{CM2}.
In \cite{Pf}, Pflueger proved that for general $k$-gonal curves,
\begin{equation} \label{theirformula}
\dim W^r_d(C) \leq \rho_k(g, r, d) := \max_{\ell \in \{0, \ldots, r'\}} \rho(g, r - \ell, d) - \ell k,
\end{equation}
where $r' := \min\{r, g-d+r -1\}$.
Jensen-Ranganathan \cite{JR} then established the existence of a component of the maximum possible dimension, thereby determining $\dim W^r_d(C) = \rho_k(g, r, d)$. 
However, many questions about its geometry remain: What are the dimensions of all components? How many components are there? Where are they smooth?
This paper determines the dimensions of all irreducible components and shows that they are smooth away from ``further degenerate loci." The key insight is that splitting loci capture more precise information that yields better control over the geometry of $W^r_d(C)$. 

We work on the Hurwitz space $\H_{k,g}$ parametrizing smooth degree $k$, genus $g$ covers $f\colon  C \rightarrow \pp^1$ over an arbitrary algebraically closed field $\cc$.
Given a line bundle $L$ on such a curve $C$, the push forward $f_* L$ is a rank $k$ vector bundle on $\pp^1$. 
By Riemann-Roch, the degree of the push forward is
\begin{equation} \label{rr}
\deg(f_*L) = \chi(C, L) - k = d - g + 1 - k.
\end{equation}
Every vector bundle on $\pp^1$ is isomorphic to $\O(e_1) \oplus \cdots \oplus \O(e_k)$ for some collection of integers $\vec{e} = (e_1, \ldots, e_k)$ with $e_1 \leq \cdots \leq e_k$.
We call such a collection $\vec{e}$ the \textit{splitting type} and abbreviate the corresponding sum of line bundles by $\O(\vec{e})$.
We then define \textit{Brill-Noether splitting loci} by
\[\Sigma_{\vec{e}}(C, f) := \{L \in \Pic^d(C) : f_*L \cong \O(\vec{e})\}. \]
Geometrically, line bundles in $\Sigma_{\vec{e}}(C, f)$ correspond to maps of $C$ into the rational scroll $\pp \O(\vec{e})^\vee \rightarrow \pp^1$ compatible with $f$.
The \textit{expected codimension} of $\Sigma_{\vec{e}}(C, f)$ is defined as
\[u(\vec{e}) := h^1(\pp^1, End(\O(\vec{e}))) = \sum_{i < j} \max \{0, e_j - e_i - 1\}.\]

The specialization of splitting types follows certain rules. Given two splitting types $\ee = (e_1', \ldots, e_k')$ and $\e = (e_1, \ldots, e_k)$, we define a partial ordering by $\ee \leq \e$ if $\O(\e)$ can specialize to $\O(\ee)$, that is if $e_1' + \ldots + e_j' \leq e_1 + \ldots + e_j$ for all $j$ (see Section \ref{split}).
We define \textit{Brill-Noether splitting degeneracy loci} by
\begin{align*}
\Sigmabar_{\vec{e}}(C, f) &:= \bigcup_{\ee \leq \e} \Sigma_{\vec{e}}(C, f).
\end{align*}
Let $H = f^*\O_{\pp^1}(1)$ be the distinguished $g^1_k$ on $C$. One may readily check
\begin{align*}
\Sigmabar_{\vec{e}}(C, f) &= \{L \in \Pic^d(C) : h^0(C, L(mH)) \geq h^0(\pp^1, \O(\vec{e})(m)) \text{ for all } m\}.
\end{align*}
Thus, splitting degeneracy loci provide a refinement of the stratification of $\Pic^d(C)$ by Brill-Noether loci $W^r_d(C)$.
In particular,
\begin{equation} \label{W}
W^r_d(C) = \bigcup_{\substack{|\vec{e}| = d - g + 1- k \\ h^0(\O(\vec{e})) \geq r+1}} \Sigmabar_{\vec{e}}(C, f) = \bigcup_{\substack{\vec{e} \text{ maximal:} \\ h^0(\O(\vec{e})) \geq r+ 1}}\Sigmabar_{\vec{e}}(C, f),
\end{equation}
where $|\vec{e}| = e_1 + \ldots + e_k$. The maximal splitting types among those with $r+1$ global sections have  a ``balanced plus balanced" shape and are uniquely determined by the number of nonnegative summands.
 When the rank and degree are understood, we define $\vec{w}_{r,\ell}$ to be the splitting type with $r+1 - \ell$ nonnegative parts that is maximal among those with $r+1$ global sections (see Lemma \ref{wlem}). In terms of these splitting loci, we can rewrite Pfleuger's formula \eqref{theirformula} suggestively to see
 \[W^r_d(C) = \bigcup_{\ell} \Sigmabar_{\vec{w}_{r,\ell}}(C, f) \qquad \text{and} \qquad \rho_k(g, r, d) = \max_{\ell} \left( g - u(\vec{w}_{r,\ell}) \right).\]

The following example demonstrates the more subtle geometry splitting loci capture.
\begin{Example} \label{Geoff}
Suppose $f\colon  C \rightarrow \pp^1$ is a general trigonal curve of genus $5$. 
By \eqref{rr}, the push forward of a degree $4$ line bundle on $C$ is a rank $3$, degree $-3$ vector bundle on $\pp^1$.
The diagram below describes the partial ordering of splitting types in the stratification of $\Pic^4(C)$ by Brill-Noether splitting loci.
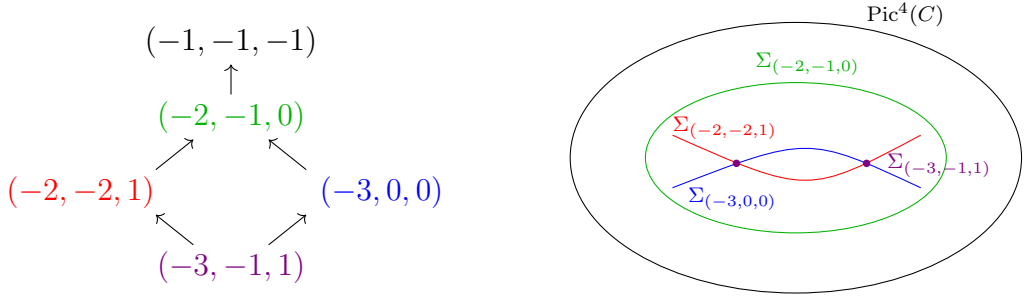
\begin{center}
\begin{tikzpicture}
\draw (0, 0) node {$(-1, -1, -1)$};
\draw [<-] (0, -.3) -- (0, -.7);
\draw [black!30!green] (0, -1) node {$(-2, -1, 0)$};
\draw [red] (-2, -2) node {$(-2, -2, 1)$};
\draw [blue] (2, -2) node {$(-3, 0, 0)$};
\draw [<-] (-1, -2-.3) -- (-.5, -2-.7);
\draw [<-] (1, -2-.3) -- (.5, -2-.7);
\draw [<-] (-.5, -1-.3) -- (-1, -1-.7);
\draw [<-] (.5, -1-.3) -- (1, -1-.7);
\draw [violet] (0, -3) node {$(-3, -1, 1)$};
\end{tikzpicture}
\hspace{.5in}
\begin{tikzpicture}
\draw [black!30!green] (3, 4+.75+1) node {\tiny $\Sigma_{(-2,-1,0)}$};
\draw (2.84,4.5) ellipse (3cm and 1.8cm);
\draw [black!30!green] (2.84,4.5) ellipse (2cm and 1cm);
\draw [red] (1.2,4.8) .. controls (3, 4) .. (4.5, 4.8);
\draw [red] (1.9, 4.9) node {\tiny $\Sigma_{(-2,-2,1)}$};
\draw [blue] (1.2,4.1) .. controls (3, 4.8) .. (4.5, 4.1);
\draw [blue] (2, 3.95) node {\tiny $\Sigma_{(-3,0, 0)}$};
\draw [violet] (2.05, 4.425) node [circle,fill,inner sep=1pt]{};
\draw [violet] (3.78, 4.425) node [circle,fill,inner sep=1pt]{};
\draw [violet] (4.75, 4.4) node {\tiny $\Sigma_{(-3, -1, 1)}$};
\draw (4.3, 6.4) node {\tiny $\Pic^4(C)$};
\end{tikzpicture}
\end{center}
We have
\[\Sigmabar_{(-1, -1, -1)}(C, f) = \Pic^4(C) \qquad \text{and} \qquad \Sigmabar_{(-2, -1, 0)}(C, f) = W^0_4(C).\]
Meanwhile, the Brill-Noether locus $W^1_4(C)$ consists of two (intersecting) components, which are distinguished by splitting type of the push forward.
 If $H$ is the trigonal class and $K$ is the canonical divisor, we have
\begin{align*}
\Sigmabar_{(-2, -2, 1)}(C, f) &= \{L \in \Pic^4(C) : h^0(C, L(-H)) \geq 1\} = \{H + p\}_{p \in C} \\
\Sigmabar_{(-3, 0, 0)}(C, f) &=  \{L \in \Pic^4(C) : h^0(C, L(H))  \geq 4\} = \{K - H - q\}_{q \in C}.
\end{align*}
We recognize these splitting types as $\vec{w}_{1,1} = (-2, -2, 1)$ and $\vec{w}_{1,0} = (-3, 0, 0)$.
Finally, the splitting locus $\Sigmabar_{(-3, -1, 1)}(C, f)$ is the intersection of the two curves above. 
If $p_0$ and $q_0$ denote the unique pair of points such that $p_0 + q_0 \sim K - 2H$, 
the intersection consists of the two points $H + p_0 \sim K - H - q_0$ and $H + q_0 \sim K - H - p_0$. This example will be revisited in Example \ref{revisit}.
\end{Example}

Our main result determines the dimensions and smoothness of all Brill-Noether splitting loci for general degree $k$ covers.
\begin{theorem} \label{main}
Let $f\colon  C \rightarrow \pp^1$ be a general degree $k$, genus $g$ cover.
Let $d$ be any integer and let $\vec{e} = (e_1, \ldots, e_k)$ be a collection of integers with $e_1 + \ldots + e_k = d - g - k + 1$.
If $u(\vec{e}) \leq g$, then $\Sigma_{\vec{e}}(C, f)$ is smooth of pure dimension $g - u(\vec{e})$. If $u(\vec{e}) > g$ then $\Sigma_{\vec{e}}(C, f)$ is empty.
\end{theorem}

The case $k = 2$ is a classical result of Clifford. (In this case, $\Sigmabar_{\vec{e}}(C, f)$ is the image of a certain symmetric power of $C$ in $\Pic^d(C)$.) We therefore assume for the rest of the paper that $k > 2$.
\begin{remark}
The Hurwitz space $\H_{k,g}$ is only known to be irreducible when the characteristic of the ground field is greater than $k$. If the characteristic is less than or equal to $k$, then by ``general degree $k$ cover" we shall mean a general deformation of our degeneration, so the conclusion holds for general $C \to \pp^1$ in some component of $\H_{k,g}$.
\end{remark}

\begin{remark}
Since it has the expected codimension, the class of $\Sigmabar_{\vec{e}}(C, f)$ is determined by the universal splitting degeneracy formulas of \cite{L}
 (see  Example \ref{revisit}).
 \end{remark}

\vspace{.1in}
\noindent
\textbf{Key ideas of the proof.}
\begin{enumerate}

\item To prove that our splitting loci are smooth of the expected dimension,
it suffices to show that for general $C$, the natural map
\[\Def^1(L) = H^1(\O_C) \rightarrow H^1(End(f_*L)) = \Def^1(f_*L)\]
is surjective for all $L \in \Pic^d(C)$.
Equivalently, by Serre duality, we seek to show that
\begin{equation} \label{sdd}
H^0(End(f_*L)(-p-q)) \to H^0(\omega_C)
\end{equation}
is injective, where $p, q$ are points on $\pp^1$
(whose sum gives the dualizing sheaf).

\item To accomplish this, we degenerate $C$ to a chain of elliptic curves, each mapping with degree $k$ to a chain of $\pp^1$'s.
On such reducible curves, \eqref{sdd} is \emph{not} necessarily injective.

\item To solve this problem,
we give an explicit description
of which endomorphisms of $f_* L$ deform with the curve.
We then apply this description to the kernels of \eqref{sdd} to show
that none of these sections deform with a general smoothing.

\item To show that the expected splitting loci are non-empty, we prove their expected classes are non-zero.
Using the universal splitting degeneracy formulas in \cite{L}, we prove
\[[\Sigmabar_{\vec{e}}(C, f)] = a_{\vec{e}} \theta^{u(\vec{e})},\]
where $a_{\vec{e}} \in \qq$ depends only on $\vec{e}$ \textit{and not on $g$}.
In general, the formulas for $a_{\vec{e}}$ are intractible to compute directly.
Instead, we deduce $a_{\vec{e}} \neq 0$ for $g$ sufficiently large by calculating $a_{\ee}$ for suitably chosen specialization $\ee$ of $\vec{e}$ where the formulas become simple.
\end{enumerate}
\vspace{.1in}

As a special case, Theorem \ref{main} determines the dimensions of all components of $W^r_d(C)$, thereby answering Question 1.12 of \cite{Pf}, and giving new proofs of the theorems in \cite{JR, Pf}.
\begin{Corollary} \label{cor}
Let $C$ be a general $k$-gonal curve of genus $g$. Let $\vec{w}_{r, \ell}$ denote the rank $k$, degree $d - g +k-1$ splitting type with $r+1 - \ell$ nonnegative parts that is maximal among those with $r+1$ global sections.
Every component of $W^r_d(C)$ is generically smooth of dimension $g - u(\vec{w}_{r,\ell}) = \rho(g, r - \ell, d) - \ell k $ for some $\max\{0, r + 2 - k\} \leq \ell \leq r$ such that $\ell = 0$ or $\ell \leq  g - d + 2r + 1 - k$.
Such a component exists for each $\ell$ with $g - u(\vec{w}_{r,\ell})  \geq 0$.
\end{Corollary}

In other words, splitting loci explain the different dimensions of components of $W^r_d(C)$ when $C$ is a general $k$-gonal curve. 
For example, when $f\colon  C \rightarrow \pp^1$ is a general trigonal curve of genus $6$, we have $\vec{w}_{2,1} = (-4, 0, 0)$ and $\vec{w}_{2,2} = (-3, -2, 1)$, so
\[W^1_4(C) = \Sigmabar_{(-4, 0, 0)}(C, f) \cup \Sigmabar_{(-3, -2, 1)}(C, f).\]
We have $u(-4, 0, 0) = 6$ so $\dim \Sigmabar_{(-4, 0, 0)} = 0$. This corresponds to the isolated $g^1_4$ associated to $C \subset \pp^1 \times \pp^1$ as a curve of bidegree $(3, 4)$. On the other hand, $u(-3, -2, 1) = 5$ so $\dim \Sigmabar_{(-3, -2, 1)} = 1$, corresponding to the $g^1_3$ plus any base point. 

\begin{remark}
Upon completing this manuscript, the author learned that Cook-Powell--Jensen have a simultaneous and independent proof that $\Sigma_{\vec{w}_{r,\ell}}(C, f)$ has a component of the expected dimension \cite{CPJ}.
\end{remark}

This paper is organized as follows. In Section \ref{split}, we recall the splitting behavior of families of vector bundles on $\pp^1$ and describe its application to Brill-Noether splitting loci. Assuming Theorem \ref{main}, Corollary \ref{cor} follows from the combinatorial structure of splitting loci stratifications. Section \ref{dimbounds} bounds the dimension of $\Sigma_{\vec{e}}(C, f)$ from above by considering a degeneration to a chain of elliptic curves. Further analysis on this degeneration yields a proof of smoothness in Section \ref{smoothness}.
 Finally, in Section \ref{existence}, we prove existence of Brill-Noether splitting loci by showing enumerative formulas for their expected classes are non-zero.

\subsection*{Acknowledgements} This work was inspired by Geoffrey Smith, who introduced the notion of Brill-Noether splitting loci in a seminar at Stanford and asked if the author's results in \cite{L} could be applied to show their existence. I am grateful for his insight. Thanks also to Ravi Vakil, Eric Larson, Sam Payne, and Melanie Wood for fruitful discussions.  I thank Kaelin Cook-Powell and Dave Jensen for their generosity and openness in sharing their work. I am grateful to the Hertz Foundation Graduate Fellowship, NSF Graduate Research Fellowship, Maryam Mirzakhani Graduate Fellowship, and the Stanford Graduate Fellowship for their generous support.

\section{Splitting loci} \label{split}
Let $B$ be a finite type scheme over a field and $\pi: \pp^1 \times B \rightarrow B$ the projection. 
Given a vector bundle $E$ on $\pp^1 \times B$, the base $B$ is stratified by \textit{splitting loci} of $E$, defined by
\[\Sigma_{\vec{e}}(E) := \{b \in B : E|_{\pi^{-1}(b)} \cong \O(\vec{e})\}.\]
The list of integers $h^0(\pp^1, \O(\vec{e})(m))$ for all $m \in \zz$ determines the splitting type $\vec{e}$: In fact, the multiplicity of $\O_{\pp^1}(-j)$ as a summand of $E$ is equal to the second difference function evaluated at $j$ of the Hilbert function $m \mapsto h^0(\pp^1, E(m))$ (see e.g. \cite[Lemma 5.6]{ES}).
In families, uppersemicontinuity of cohomology on fibers of $\pi$ constrains which splitting types can specialize to others. Given two splitting types $\e = (e_1, \ldots, e_k)$ and $\ee = (e_1', \ldots, e_k')$, we write $\ee \leq \e$ if $e_1'+\ldots + e_j' \leq e_1+\ldots+e_j$ for all $j$. 
For each rank and degree, there is a unique maximal splitting type called the \textit{balanced} splitting type, which is characterized by the condition $|e_i - e_j| \leq 1$ for all $i, j$. We denote the \textit{balanced bundle} of rank $r$ and degree $d$ by $B(r, d)$.
We define \textit{splitting degneracy loci} by
\begin{equation} \label{u}
\overline{\Sigma}_{\vec{e}}(E) := \bigcup_{\ee \leq \e} \Sigma_{\vec{e}}(E).
\end{equation}

Recall that the \textit{expected codimension} of $\Sigma_{\vec{e}}(E)$ is
\[u(\vec{e}) := h^1(\pp^1, End(\O(\vec{e}))) = \sum_{i < j} \max \{0, e_j - e_i - 1\},\]
which is the dimension of the deformation space of $\O(\vec{e})$.
In general, $\overline{\Sigma}_{\vec{e}}(E)$ is always closed, but need not be the closure of $\Sigma_{\vec{e}}(E)$.
However, in the case that all splitting loci have the expected dimension, the following lemma shows $\overline{\Sigma}_{\vec{e}}(E)$ is the closure of $\Sigma_{\vec{e}}(E)$. Thus, no confusion should result from this notation.
\begin{Lemma} \label{low}
Let $E$ be a vector bundle on $\pi: \pp^1 \times B \rightarrow B$ with $B$ irreducible. If $\Sigmabar_{\vec{e}}(E)$ is non-empty, then every component of $\Sigmabar_{\vec{e}}(E)$ has at least the expected dimension. In particular, if all $\Sigma_{\vec{e}}(E)$ have the expected dimension, then $\overline{\Sigma}_{\vec{e}}(E)$ is the closure of $\Sigma_{\vec{e}}(E)$.
\end{Lemma}
\begin{proof}
Let $\E$ be the universal bundle over the moduli stack $\B$ of vector bundles on $\pp^1$ bundles. Then $\Sigmabar_{\vec{e}}(\E)$ has codimension $u(\vec{e})$ and
$\Sigmabar_{\vec{e}}(E)$ is its preimage under the induced map $B \rightarrow \B$. Codimension can only decrease under pullback so $\codim \Sigmabar_{\vec{e}}(E) \leq u(\vec{e})$. This applies on any open set of $B$, so every component of $\Sigmabar_{\vec{e}}(E)$ has at least the expected dimension.
If all splitting loci have the expected dimension, every component of $\Sigmabar_{\vec{e}}(E) \backslash \Sigma_{\vec{e}}(E)$ has dimension less than the expected dimension of $\Sigmabar_{\vec{e}}(E)$. Thus, all of $\Sigmabar_{\vec{e}}(E) \backslash \Sigma_{\vec{e}}(E)$ must lie in the closure of $ \Sigma_{\vec{e}}(E)$.
\end{proof}

We now realize the Brill-Noether splitting loci defined in the introduction as splitting loci of a vector bundle on $\pp^1 \times \Pic^d(C)$.
Let $f\colon  C \rightarrow \pp^1$ be a degree $k$, genus $g$ cover and consider the following commuting triangle
\begin{center}
\begin{tikzcd}
&C \times \Pic^d(C) \arrow{r}{f \times \mathrm{id}} \arrow{rd}[swap]{\nu} & \pp^1 \times \Pic^d(C) \arrow{d}{\pi}\\
& & \Pic^d(C).
\end{tikzcd}
\end{center}
Let $\mathcal{L}$ be a Poincar\'e line bundle on $C \times \Pic^d(C)$, that is, a line bundle with the property that $\L|_{\nu^{-1}[L]} \cong L$ (see e.g. \cite[\S IV.2]{ACGH}).
The push forward $\E := (f \times \mathrm{id})_*\L$ is a vector bundle on $\pp^1 \times \Pic^d(C)$ with the property that $\E|_{\pi^{-1}[L]} \cong f_*L$.
In other words, the Brill-Noether splitting loci defined in the introduction are the splitting loci of $\E$:
\[\Sigma_{\vec{e}}(C, f) = \Sigma_{\vec{e}}(\E) \qquad \text{and} \qquad \Sigmabar_{\vec{e}}(C, f) = \overline{\Sigma}_{\vec{e}}(\E).\]
 By Riemann-Roch, the degree of $\E$ on a fiber of $\pi$ is
\[
\deg \E|_{\pi^{-1}[L]} = \chi(C, \E|_{\pi^{-1}[L}]) - \rank \E|_{\pi^{-1}[L]} = \chi(L) - k = d - g + 1 - k.
\]

It follows from the definitions that
\begin{equation} \label{weq}
W^r_d(C) = \bigcup_{\substack{\vec{e} \\ h^0(\O(\vec{e})) \geq r+1}} \Sigma_{\vec{e}}(C, f) = \bigcup_{\substack{\vec{e} \text{ maximal}\\ h^0(\O(\vec{e})) \geq r+1}} \overline{\Sigma}_{\vec{e}}(C, f).
\end{equation}
That is, to characterize contributions of splitting loci to $W^r_d(C)$ we are interested in splitting types that are maximal with respect to the partial ordering among those satisfying $h^0(\O(\vec{e}))\geq r+1$.

\begin{Lemma} \label{wlem}
Let $d' = d - g + 1 - k$ and suppose $r > d - g$.
The maximal splitting types of rank $k$, degree $d'$ among those satisfying $h^0(\O(\vec{e}))\geq r+1$ 
are 
\[\vec{w}_{r,\ell} := B(k-r-1+\ell, d'-\ell) \oplus B(r+1-\ell, \ell)\]
for $\max\{0, r+2-k\} \leq \ell \leq r$ such that $\ell = 0$ or $\ell \leq g - d + 2r + 1 - k$. Moreover,
\[u(\vec{w}_{r, \ell}) = g - \rho(g, r-\ell, d) + \ell k.\]
\end{Lemma}
\begin{remark}
If $r \leq d - g$ we automatically have $W^r_d(C) = \Pic^d(C)$.
As Pflueger points out in \cite[Remarks 1.6 and 3.2]{Pf},  the codimension $g - \rho(g, r-\ell, d) + \ell k$ is quadratic in $\ell$, achieving its minimum at $\ell_0 = \frac{1}{2}(g - d +2r + 1-k)$.
Our lower bound $r + 2 - k$ is the same distance from the minimum $\ell_0$ as Pflueger's upper bound $g - d + r - 1$. From this, it is not hard to see that the minimum over $\ell$ in our range is the same as Pflueger's minimum.
\end{remark}
\begin{proof}
The assumption $r > d - g$ implies $k - r - 1 + \ell < \ell - d'$ so $B(k-r-1+\ell, d'-\ell)$ consists of entirely negative summands.
Requiring that the rank of this vector bundle is positive gives our lower bound $\ell \geq r + 2 - k$.

First we show every $\vec{e}$ with $h^0(\O(\vec{e})) \geq r+1$ is less than $\vec{w}_{r,\ell}$ for some $\ell$.
We may write $\O(\vec{e}) = N \oplus P$ where $N$ consists of negative summands, and $P$ consists of nonnegative summands. If $h^0(\pp^1, P) > r + 1$, then the splitting type obtained from $\vec{e}$ by decreasing the largest summand by one and increasing the lowest summand by one is more balanced than $\vec{e}$ and still has at least $r+1$ sections. Hence, it suffices to consider the case $h^0(\pp^1, P) = r +1$. Then, $\vec{e} \leq \vec{w}_{r,\ell}$ for $\ell = \deg P$.

By construction, the only splitting types more balanced than $\vec{w}_{r,\ell}$ are obtained from $\vec{w}_{r,\ell}$ by lowering a summand in $B(r+1 - \ell, \ell)$ and raising a summand in $B(k-r-1+\ell, d'-\ell) $. This produces a splitting type with less than $r+1$ global sections unless $\ell > 0$ and $B(k-r-1+\ell, d'-\ell)$ has a summand of degree $-1$. In that case, we see $\vec{w}_{r,\ell} < \vec{w}_{r, \ell-1}$. Thus, $\vec{w}_{r,\ell}$ is maximal precisely when $\ell = 0$ or all summands of $B(k-r-1+\ell, d'-\ell)$ are degree at most $-2$. The latter means
$2(k - r - 1 + \ell) \leq \ell - d'$, which is equivalent to $\ell \leq g - d + 2r + 1 - k$.

Finally, the expected codimension of $\vec{w}_{r,\ell}$ is 
\begin{align*}
u(\vec{w}_{r,\ell}) &= h^1(\pp^1, End(\vec{w}_{r,\ell})) = h^1(\pp^1, Hom(B(r+1 - \ell, \ell), B(k - r - 1 + \ell, d' - \ell)) \\
&= - \chi(\pp^1, Hom(B(r+1 - \ell, \ell), B(k - r - 1 + \ell, d' - \ell)))\\
&= \ell(k - r - 1 + \ell) - (d' - \ell)(r+1 - \ell) - (r + 1 - \ell)(k - r - 1 + \ell) \\
&= \ell k -(r + 1 - \ell)( d -g - r + \ell). \qedhere
\end{align*}
\end{proof}

\begin{Example}
The following table lists the ``balanced plus balanced" splitting types of rank $5$ and degree $-4$ with at least $4$ global sections. The first three are maximal.
\begin{center}
\begingroup
\setlength{\tabcolsep}{8pt} % Default value: 6pt
\renewcommand{\arraystretch}{1.5}
\begin{tabular}{c||c|c|c|c}
$\ell$ & 0 & 1 & 2 & 3 \\ [2pt]
\hline
\hline
$\vec{w}_{3,\ell}$ & $(-4, 0, 0, 0, 0)$ & $(-3, -2, 0, 0, 1)$ & $(-2, -2, -2, 1, 1)$ & $(-2, -2, -2, -1, 3)$ \\ [2pt]
\hline
$u(\vec{w}_{3,\ell})$ & $12$ & $11$ & $12$ & $15$
\end{tabular}
\endgroup
\end{center}
\vspace{.1in}
Notice that $w_{3,3} < w_{3,2}$ in the partial ordering, showing necessity of the condition $\ell \leq g - d + 2r + 1 - k$ in Lemma \ref{wlem}. 
Corollary \ref{cor} says that for a general pentagonal curve, every component of $W^3_g(C)$ has dimension $g - 11$ or $g - 12$. Moreover, there is at least one component of dimension $g - 11$ and at least  two components of dimension $g - 12$ when these quantities are nonnegative.
\end{Example}

Assuming Theorem \ref{main}, Corollary \ref{cor} now follows.

\begin{proof}[Proof of Corollary \ref{cor}]
Equation \eqref{weq} and Lemma \ref{wlem} show that $W^r_d(C)$ is the union of $\Sigmabar_{\vec{w}_{r,\ell}}(C, f)$ for $\max\{0, r+2 - k\} \leq \ell \leq r$
such that $\ell = 0$ or $\ell \leq g - d + 2r + 1 - k$. Theorem \ref{main} asserts that $\Sigma_{\vec{w}_{r,\ell}}(C, f)$ is smooth of pure dimension $g - u(\vec{w}_{r,\ell})$ whenever this quantity is nonnegative, and Lemma \ref{low} guarantees that $\Sigmabar_{\vec{w}_{r,\ell}}(C, f)$ is its closure.
\end{proof}

\section{The degeneration and dimension bounds} \label{dimbounds}
In this section, we describe our degeneration to a chain of elliptic curves and prove a smoothing theorem for endomorphisms of the push forwards of line bundles. This involves explicit compatibility conditions at the nodes, in a manner similar to Eisenbud-Harris' proof of the Brill-Noether theorem \cite{EH} (see also \cite[Ch. 5]{HM} for an exposition). A noteworthy difference in the set up is that elliptic curves in the middle of our chain have more than one node, creating subtleties in how these conditions interact.
Also, instead of tracking vanishing sequences of different limit line bundles, we describe the sections that smooth from a fixed limit.

 A consequence of our analysis will be that
\begin{equation} \label{equiv}
\dim\{L \in \Pic^d(C) : h^0(\pp^1, End(f_*L)) \geq \delta+k^2\} \leq g - \delta.
\end{equation}
for all $\delta \geq 0$. 
Since $End(f_*L)$ is rank $k^2$ and degree $0$ on $\pp^1$,
\[u(f_*L) = h^1(\pp^1, End(f_*L)) =  h^0(\pp^1, End(f_*L)) - k^2,\]
so \eqref{equiv} implies $\dim \Sigma_{\vec{e}}(C, f) \leq g - u(\vec{e})$ for all $\vec{e}$.
(Notice that \eqref{equiv} does not refer to a particular splitting type!)

Basic cohomological observations determine all push forwards of line bundles from elliptic curves.
\begin{Lemma} \label{ezpush}
Let $X$ be an elliptic curve and $f\colon  X \rightarrow \pp^1$ a degree $k$ map. Let $L$ be a line bundle of degree $d = a + nk$ on $X$ with $0 \leq a < k$. We have
\[f_*L = \begin{cases}  \O_{\pp^1}(n-2) \oplus \O_{\pp^1}(n-1)^{\oplus k-2} \oplus  \O_{\pp^1}(n) &\text{if $L = f^*\O_{\pp^1}(n)$} \\  \O_{\pp^1}(n-1)^{\oplus k-a} \oplus \O_{\pp^1}(n)^{\oplus a} &\text{otherwise.} \end{cases}\]
 \end{Lemma}
\begin{proof}
Let $H = f^*\O_{\pp^1}(1)$. By the projection formula, $f_* L = \O_{\pp^1}(n) \otimes f_* L(-nH)$, so it suffices to consider the case $n = 0$.
First observe that $h^0(X,\O_X) = h^1(X,\O_X) = 1$. The only rank $k$ vector bundle on $\pp^1$ with this cohomology is $\O_{\pp^1}(-2) \oplus \O_{\pp^1}(-1)^{\oplus k-2} \oplus \O_{\pp^1}$ so this must be $f_*\O_X$, completing the first case.
Now suppose $L$ is non-trivial of degree $0 \leq a < k$. By Serre duality, $h^1(X,L) = h^0(X,L^\vee) = 0$, implying all summands of $f_*L$ are degree at least $-1$. Riemann-Roch shows $h^0(X,L) = a$ and moreover, $h^0(X,L(-H)) = 0$, because in this case $\deg L(-H) < 0$. It follows that $f_*L = \O(-1)^{\oplus k-a} \oplus \O_{\pp^1}^{\oplus a}$, completing the second case.
\end{proof}

We now describe our degeneration. Let $\X_0$ be a simply nodal chain of elliptic curves $X_1, \ldots, X_g$ with $X_i$ joined to $X_{i+1}$ at a point $p_i$.
For each $i$, let $f_i: X_i \rightarrow \pp^1$ be a degree $k$ map that is totally ramified at $p_{i-1}$ and $p_i$. 
Note then that $p_{i-1}$ and $p_{i}$ differ by a $k$-torsion element on $X_i$.
Together, the $f_i$ define a map of $\X_0$ to a simply nodal chain $\P_0$ of $\pp^1$'s, labeled $P_1, \ldots, P_g$, with nodes $f_i(p_i) = f_{i+1}(p_i)$ for $1 \leq i \leq g-1$.

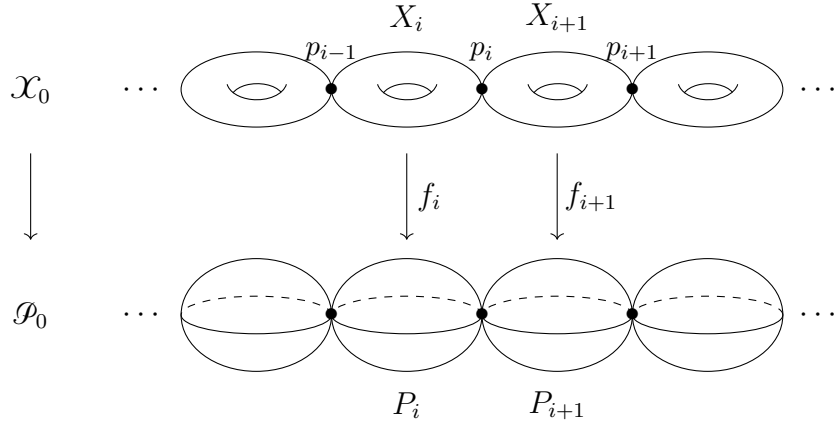
\vspace{.15in}
\begin{center}
\begin{tikzpicture}
\draw (0,0) ellipse (1cm and .5cm);
 \draw (.4,.07) arc(-10:-170:.4cm and .25cm);
 \draw (.33,-.038) arc(25:158:.35cm and .2 cm);
\draw (-1, 0) node {\small $\bullet$};
\draw (1, .5) node {\small $p_{i+1}$};
\draw (-1, .5) node {\small $p_i$};
\draw (-3, .5) node {\small $p_{i-1}$};
\draw (-7, 0) node {$\X_0$};
\draw (-7, -3) node {$\P_0$};
\draw [->] (-7, -.85) -- (-7, -2);
%\draw (-7-.2, -1 -.425) node {$\f$};
\draw (-2, .95) node{$X_i$};
\draw (0, .95) node{$X_{i+1}$};
\draw [->] (-2, -.85) -- (-2, -2);
\draw (-1.7, -1 -.425) node {$f_i$};
\draw [->] (-2+2, -.85) -- (-2+2, -2);
\draw (-1.7+2.15, -1 -.425) node {$f_{i+1}$};
\draw (-2, -4.2) node{$P_i$};
\draw (0, -4.2) node{$P_{i+1}$};
\draw (2,0) ellipse (1cm and .5cm);
 \draw (2.4,.07) arc(-10:-170:.4cm and .25cm);
 \draw (2.33,-.038) arc(25:158:.35cm and .2 cm);
\draw (1, 0) node {\small $\bullet$};
\draw (1, -3) node {\small $\bullet$};
\draw (-1, -3) node {\small $\bullet$};
\draw (-3, -3) node {\small $\bullet$};
\draw (-2,0) ellipse (1cm and .5cm);
 \draw (.4 -2,.07) arc(-10:-170:.4cm and .25cm);
 \draw (.33-2,-.038) arc(25:158:.35cm and .2 cm);
 \draw (-4,0) ellipse (1cm and .5cm);
 \draw (.4 -4,.07) arc(-10:-170:.4cm and .25cm);
 \draw (.33-4,-.038) arc(25:158:.35cm and .2 cm);
\draw (-3, 0) node {\small $\bullet$};
\draw (-5.5, 0) node {$\cdots$};
\draw (3.5, 0) node {$\cdots$};
\draw (-5.5, -3) node {$\cdots$};
\draw (3.5, -3) node {$\cdots$};
\draw (0, -3) ellipse (1cm and .75cm);
 \draw (1,-3) arc(0:-180:1cm and .25cm);
  \draw[dashed] (1,-3) arc(0:180:1cm and .25cm);
  \draw (2, -3) ellipse (1cm and .75cm);
 \draw (1+2,-3) arc(0:-180:1cm and .25cm);
  \draw[dashed] (1+2,-3) arc(0:180:1cm and .25cm);
    \draw (-2, -3) ellipse (1cm and .75cm);
 \draw (1-4,-3) arc(0:-180:1cm and .25cm);
  \draw[dashed] (1-4,-3) arc(0:180:1cm and .25cm);
   \draw (-4, -3) ellipse (1cm and .75cm);
    \draw (1-2,-3) arc(0:-180:1cm and .25cm);
  \draw[dashed] (1-2,-3) arc(0:180:1cm and .25cm);
\end{tikzpicture}
\end{center}
By the theory of admissible covers \cite{admiss}, the nodes smooth to obtain a family $\f\colon  \mathcal{X} \rightarrow \mathcal{P}$, flat over a pointed smooth curve $(\Delta, 0)$, where the general fiber is a smooth curve of genus $g$ mapping to $\pp^1$ and the central fiber is the map $\X_0 \rightarrow \P_0$ restricting to $f_i$ on each component $X_i$. (For a treatment in positive characteristic see Section 5 of \cite{liu}.)
 By slight abuse of notation, we also write $\f$ for this map on the central fiber.
 We write $\tilde{\X} \rightarrow \tilde{\P}$ for the family over the punctured curve $\Delta \backslash \{0\}$.

The curve $\X_0$ is compact type. In particular, given a degree $d$ line bundle $\tilde{\L}$ on $\tilde{\X}$ and 
 partition $d = d_1 + \ldots + d_g$, there is a unique extension $\mathcal{L}$ to $\X$ so that the limit $\L_0 = \L|_{\X_0}$ restricts to a degree $d_i$ line bundle $L_i$ on $X_i$.  We wish to bound $h^0(\P_t, End(\f_*\L_t))$ for general $t$. To do so, we study the subspace of $V_\L \subset H^0(\P_0, End(\f_*\L_0))$ of sections which arise as limits of sections in $H^0(\P_t, End(\f_*\L_t))$ as $t \rightarrow 0$.

 For simplicity, let us fix a partition $d_1 + \ldots + d_g = d$ and
 write 
 \[\Pic^d(X) := \Pic^{d_1}(X_1) \times \Pic^{d_2}(X_2) \times \cdots \times \Pic^{d_g}(X_g).\]  
 Let $\alpha_i : X_i \rightarrow \X_0$ denote the inclusion.
 Given $\L_0 = (L_1, \ldots, L_g) \in \Pic^d(\X_0)$, we have a short exact sequence on $\X_0$
 \[0 \rightarrow \L_0 \rightarrow \bigoplus_{i=1}^g (\alpha_i)_* L_i \rightarrow \bigoplus_{i=1}^{g-1} \O_{p_i} \rightarrow 0.\] 
Let $\beta_i : P_i \rightarrow \P_0$ be the inclusion.
For each $i$, let $E_i = (f_i)_* L_i$.
Applying $\f_*$ to the above, we obtain an exact sequence on $\P_0$
 \begin{equation} \label{flseq}
 0 \rightarrow \f_*\L_0 \rightarrow \bigoplus_{i=1}^g (\beta_i)_*E_i \rightarrow \bigoplus_{i=1}^g \O_{f(p_i)} \rightarrow 0.
 \end{equation}
\begin{warning}
The map $\f|_{\X_0}: \X_0 \rightarrow \P_0$ is not flat and $\f_*\L_0$ is \text{not} locally free at the nodes $\f(p_i)$. Nevertheless, our family is flat over the curve $\Delta$, so $\f_*\L_0 = (\f_*\L)|_{\P_0}$ is the limit we wish to consider.
\end{warning}
The restriction of $\f_*\L_0$ to each component $P_i$ has an isomorphism
\begin{equation} \label{compiso}
(\f_*\L_0)|_{P_i} \cong E_i \oplus \O_{\f(p_{i-1})}^{\oplus k-1} \oplus \O_{\f(p_i)}^{\oplus k-1}.
\end{equation}
In what follows, we will write $L_i|_{kp}$ for $L_i|_{f_i^{-1}(f_i(p))}$.
Above, $\O_{\f(p_{i-1})}^{\oplus k-1}$ is identified with the subspace of $H^0(L_{i-1}|_{kp_{i-1}})$ of functions vanishing at $p_{i-1}$ and 
$\O_{\f(p_{i})}^{\oplus k-1}$ is identified with the subspace of $H^0(L_{i+1}|_{kp_i})$ of functions vanishing at $p_{i}$.
The splitting of the middle term is defined by the map sending a section $\sigma$ to $\sigma|_{kp_{i-1}} - \sigma(p_{i-1}) \in H^0(L_{i-1}|_{kp_{i-1}})$. We think of the $k-1$ factors $\O_{f(p_{i-1})}^{\oplus k-1}$ as remembering the values of the first $k-1$ derivatives of $\sigma$ at $p_{i-1}$ along $X_{i-1}$, and similarly the $k-1$ factors $\O_{f(p_{i})}$ as remembering the values of the first $k-1$ derivatives of $\sigma$ at $p_{i}$ along $X_{i+1}$.

Applying $Hom(\f_*\L_0, -)$ to \eqref{flseq} and using \eqref{compiso}, we obtain an injection of sheaves
\[End(\f_*\L_0) \hookrightarrow \bigoplus_{i = 1}^g Hom(\f_*\L_0, (\beta_i)_*E_i) \cong \bigoplus_{i=1}^g Hom((\f_*\L_0)|_{P_i}, E_i) \cong \bigoplus_{i=1}^g End(E_i).\]
The last isomorphism follows because there are no non-zero maps from the torsion summands to a locally free sheaf.
Taking global sections yields an inclusion
\[ \iota: H^0(\P_0, End(\f_*\L_0)) \hookrightarrow \bigoplus_{i=1}^g H^0(P_i, End(E_i)).\]
We want to describe the image under $\iota$ of the subspace $V_\L \subset H^0(\P_0, End(\f_*\L_0))$ of sections that arise as limits from smooth curves. One necessary condition on each factor is described in the following definition. In what follows $\cc$, denotes the ground field, which is algebraically closed of any characteristic.
\begin{Definition}
Let $L$ be a line bundle on an elliptic curve with a degree $k$ map $f\colon  X \rightarrow \pp^1$ and let $E = f_* L$. Given a point $p$ of total ramification of $f$,
we say $\phi \in H^0(\pp^1, End(E))$ is \textit{order preserving at $p$} if $\ord_p(\phi(\sigma)) \geq \ord_p \sigma$ for all $\sigma \in H^0(U, L)$ for any $U \ni p$. Equivalently, the restriction $\res \phi \in \End(H^0(L|_{kp})) \cong \End(\cc[x]/(x^k))$ is lower triangular with respect to the basis $1, x, x^2, \ldots, x^{k-1}$. 
Note that the diagonal entries $d_{p}^{(j)}(\phi) := (\res \phi)(x^j)/x^j|_{x=0}$ are independent of choice of local coordinate $x$.
\end{Definition}

We now describe agreement conditions near the nodes that are satisfied by every element of $\iota(V_\L)$.
 \begin{Lemma} \label{cond}
Given $\L_0 = (L_1, \ldots, L_g) \in \Pic^d(\X_0)$, let $\L$ be a line bundle on $\X$ such that $\L|_{\X_0} = \L_0$.
Let $V_\L \subset H^0(\P_0, End(\f_*\L_0))$ be the subspace of sections that can be extended to $H^0(\P, End(\f_*\L))$.
If $(\phi_1, \ldots, \phi_g) \in \iota(V_\L)$ then the following conditions hold for each $i=1, \ldots, g-1$:
\begin{enumerate}
\item \label{x} $\phi_i$ is order preserving at $p_i$
\item \label{y} $\phi_{i+1}$ is order preserving at $p_i$
\item \label{d} We have $d_{p_i}^{(0)}(\phi_i) = d_{p_i}^{(0)}(\phi_{i+1})$ and $d_{p_i}^{(j)}(\phi_i) = d_{p_i}^{(k-j)}(\phi_{i+1})$ for $j = 1, \ldots, k-1$.
 \end{enumerate}
 \end{Lemma}
 
 \begin{proof}
 It suffices to work locally around $p_i$. 
 Let $\cc$ be the ground field, which is algebraically closed of any characteristic.
 We may choose formal local coordinates $x, y, t$ near $p_i$ so that $\widehat{\O}_{\X, p_i} = \cc\ll x, y \rr /(xy - t)$ and $\widehat{\O}_{\P, \f(p_i)} = \cc\ll a, b \rr /(ab - t^k)$ and the map $\f$ is described locally by a map $\cc\ll a, b, t \rr /(ab - t^k) \rightarrow \cc\ll x, y, t \rr/(xy - t)$ such that $a \mapsto x^k u^{-1}$ and $b \mapsto y^k u$ for $u$ a power series in $x, y$ with constant coefficient $1$. (If the characteristic of $\cc$ does not divide $k$, we can extract a $k$th root of $u$ and absorb it into $x$ and $y$, and thereby assume $u  = 1$.)
 
 Since $\L$ is locally free, a section of $End(\f_*\L)$ is given locally near $\f(p_i)$ by an an endomorphism $\psi$ of $\cc\ll x, y, t \rr/(xy - t)$ viewed as an $\cc\ll a, b, t \rr /(ab - t^k) $ module. 
On the central fiber, the monomials $1, x, x^2, \ldots, x^{k-1}, y, y^2, \ldots, y^{k-1}$ generate $\cc\ll x, y \rr/(xy)$ as a module over $\cc\ll a, b \rr /(ab)$. By Nakayama's lemma, these monomials generate $\cc\ll x, y, t \rr/(xy - t)$ as an $\cc \ll a, b, t \rr /(ab - t^k)$ module. Because $\psi$ is a module homomorphism, we have
 \begin{align}
  (y^ku) \cdot \psi(x^j) &= b \cdot \psi(x^j) = \psi(b \cdot x^j) = \psi(y^{k-j} \cdot u \cdot (y^jx^j)) \notag \\
  &= \psi(y^{k-j} \cdot u \cdot t^j) = \psi(y^{k-j} \cdot u) \cdot t^j = \psi(y^{k-j} \cdot u) \cdot x^jy^j. \label{calc}
  \end{align}
Hence, $x^j$ divides $\psi(x^j)$. A similar argument shows that $y^{j}$ divides $\psi(y^{j})$. Thus, $\psi$ is order-preserving, so conditions \eqref{x} and \eqref{y} are satisfied.
Moreover, since $xy = t$, we see that $x^iy^j$ divides 
\[\psi(x^i y^j) = \begin{cases} t^{i} \psi(y^{j -i})  =  x^{i} y^i \psi(y^{j -i}) & \text{if $i \leq j$} \\ t^j \psi(x^{i - j}) =  x^jy^j \psi(x^{i - j}) &\text{if $j \leq i$} \end{cases}\]
for all $i, j$. Since $u = 1 + (x, y)$, it follows that $\psi(y^{k-j} \cdot u) = \psi(y^{k-j}) + y^{k-j} \cdot (x, y)$.
Dividing both sides of $\eqref{calc}$ by $x^j y^{k}$, we see that
 \[u \cdot \frac{\psi(x^j)}{x^j} = \frac{\psi(y^{k-j} \cdot u)}{y^{k-j}} = \frac{\psi(y^{k-j})}{y^{k-j}} + (x, y).\]
When $j = 1 , \ldots, k-1$, setting $x = y = 0$ in the equation above establishes part (3). (The case $d^{(0)}_{p_i}(\phi_i) = d_{p_i}^{(0)}(\phi_{i+1})$ follows from the fact that both are equal to the constant term of $\psi(1)$.)
It follows that any collection $(\phi_1, \ldots, \phi_g)$ which arrises as a limit of a section defined on smooth curves must satisfy these local compatibility properties near the nodes.
   \end{proof}
 
 Notice that conditions \eqref{x} and \eqref{y} of Lemma \ref{cond}  each represent $k(k-1)/2$ linear conditions on $\phi_i$ and $\phi_{i+1}$. Condition \eqref{d} represents another $k$ linear conditions on $\phi_i$ and $\phi_{i+1}$, for a total of $k^2$ possible linear conditions near each node. Our next task is to show that these conditions are all independent for general $(L_1, \ldots, L_g)$, and bound the dimension of the subvarieties inside $\Pic^d(\X_0)$ where they fail to be independent by a certain amount.
 The key technical lemma is to establish when the constraints on $\phi_i \in H^0(P_i, End(E_i))$ coming from the two different nodes $p_i$ and $p_{i+1}$ are independent.
   
 \begin{Lemma} \label{zing}
Suppose we have an elliptic curve with a degree $k$ map $f\colon  X \rightarrow \pp^1$ which is totally ramified over two distinct points $p, q \in X$. Let $L$ be a line bundle on $X$ which is not isomorphic to $f^*\O_{\pp^1}(n)$ for any $n$, and set $E = f_*L$. 
 Let $W_p \subset H^0(\pp^1,End(E))$ (respectively $W_q \subset H^0(\pp^1, End(E))$) denote the subspace of sections which are order preserving at $p$ (respectively $q$).
 Then, 
 \[\dim W_q = \begin{cases} \frac{k(k+1)}{2} + 1 & \text{if $L \cong \O_X(mq)$ for some $m$} \\ \frac{k(k+1)}{2} & \text{otherwise}\end{cases}\]
 and
 \[\dim W_p \cap W_q = \begin{cases} k+1 & \text{if $L \cong \O_X(n p+ m q)$ for some $m, n$} \\ k & \text{otherwise}. \end{cases} \]
 Moreover, the map $\diag: W_p \cap W_q \rightarrow \cc^{\oplus k}$ given by
 \begin{equation} \label{map}
 \phi \mapsto \left(d_{p}^{(0)}(\phi), \ldots, d_{p}^{(k-1)}(\phi)\right)
 \end{equation}
 is surjective. Finally, if $\phi \in \ker(\diag) \cap W_p \cap W_q$, then $\phi$ can be represented by a matrix with at most one non-zero entry.
 \end{Lemma}
 \begin{proof}
 The rough idea is to choose a decomposition of $E$ so that the condition of being order preserving at $p$ is that a matrix for an endomorphism is lower triangular, while the condition of being order preserving at $q$ is that the matrix is upper triangular. We shall see that if $L \ncong \O_X(np + mq)$, then the conditions to be order preserving and $p$ and at $q$ are independent, while if 
 $L \cong \O_X(np + mq)$ we obtain one less condition.
Twisting $L$ by $f^*\O_{\pp^1}(1)$ does not change $End(E)$, so we assume $k \leq \deg L < 2k$.
 
We first prove the case when $L \ncong \O_X(np + mq)$. 
By Lemma \ref{ezpush}, $E \cong \O_{\pp^1}^{\oplus k-a} \oplus \O_{\pp^1}(1)^{\oplus a}$ where $a = \deg L - k$.
Let $s, t \in f^*\O_{\pp^1}(1)$ denote sections defining the map $f$ with $V(s) = kq$ and $V(t) = kp$. For each $0 \leq j \leq a-1$, and $\alpha, \beta \in \cc$, there is a section $\tau_j(\alpha, \beta)= (\alpha s + \beta t) \cdot u_j \in H^0(X, L)$
 where $V(u_j) = jp + (a-1-j)q + r_j$. Note that $r_j \neq p, q$ by assumption. For each $j$, the $\tau_j(\alpha, \beta)$ span a copy of $H^0(\pp^1, \O_{\pp^1}(1))$ inside $H^0(\pp^1, E) = H^0(X, L)$. For $a \leq j \leq k-1$, we choose $\sigma_j \in H^0(X, L)$ so that $V(\sigma_j) =  jp + (k+a-1-j)q + r_j$, where again $r_j \neq p,q$. 
These $\sigma_j$ are non-vanishing on fibers of $f$, so each corresponds to an $H^0(\pp^1, \O_{\pp^1})$ factor inside $H^0(\pp^1, E) = H^0(X, L)$.
With respect to this decomposition of $E$, an element of $H^0(\pp^1, End(E))$ is represented by a block upper triangular matrix where the two diagonal blocks consist of elements of $\cc$ and the upper block consists of linear forms.
 \begin{equation} \label{mat}
\phi =  \left( \begin{matrix}
 c_{0,0} & \cdots & c_{0,a-1} & \alpha_{0, a} s + \beta_{0, a} t & \cdots & \alpha_{0, k-1} s + \beta_{0, k-1}t \\
  c_{1,0} & \cdots & c_{1,a-1} & \alpha_{1, a} s + \beta_{1, a} t & \cdots & \alpha_{1, k-1} s + \beta_{1, k-1}t \\
  \vdots & \ddots & \vdots & \vdots & \ddots & \vdots \\
   c_{d-1,0} & \cdots & c_{a-1,a-1} & \alpha_{a-1, a} s + \beta_{a-1, a} t & \cdots & \alpha_{a-1, k-1} s + \beta_{a-1, a-1}t \\
   0 & \cdots & 0 & c_{a,a} & \cdots & c_{a,k-1} \\
   0 & \cdots & 0 & c_{a+1,a} & \cdots & c_{a+1,k-1} \\
     \vdots & \ddots & \vdots & \vdots & \ddots & \vdots \\
        0 & \cdots & 0 & c_{k-1,a} & \cdots & c_{k-1,k-1} \\
 \end{matrix} \right)
 \end{equation}
 For $\ell \geq a$ and $j \geq a$, the coefficients $\alpha_{\ell, j}$ and $\beta_{\ell, j}$ specify which $\tau_\ell(\alpha_{\ell, j}, \beta_{\ell, j})$ appears in the image of $\sigma_j$ with respect to our chosen decomposition of $H^0(\pp^1, E)$.
 The condition for $\phi$ to be order preserving at $p$ is that $\alpha_{\ell, j} = 0$ for all $\ell, j$ and $c_{\ell, j} = 0$ for all $\ell < j$. Hence, $\dim W_p = k(k+1)/2$. The condition for $\phi$ to be order preserving at $q$ is that $\beta_{\ell, j} = 0$ for all $\ell, j$ and $c_{\ell, j} = 0$ for all $\ell > j$. It follows that $\dim W_p \cap W_q = k$ and, as the diagonal entries are unconstrained, $\diag$ is a surjection. Hence, $\ker(\diag) \cap W_p \cap W_q = \{0\}$.

Now suppose $L \cong \O_X(n p + mq) \ncong \O_X(kp)$. Without loss of generality, we may assume that $n \geq m > 0$. Since $n + m = a+ k$ with $a < k$, we also have $n > a$. Again, we have $E \cong \O_{\pp^1}^{\oplus k-a} \oplus \O_{\pp^1}(1)^{\oplus a}$, but the argument in the previous paragraph must be modified because $r_{n-1} = p$ and $r_n = q$ (or when $n = k$, we have $r_{k-1} = p$ and $r_0 = q$). Instead, for $a \leq j \leq k-1$, we choose $\sigma_j \in H^0(\pp^1, E)$ so that $\ord_p(\sigma_j) = j$ and
\[\ord_{q}(\sigma_j) = \begin{cases} k+a-n & \text{if $j = n$} \\ k + a - n - 1&\text{if $j = n-1$} \\  k+a - j - 1 & \text{otherwise.} \end{cases}\]
If $n = k$, then the first case above never occurs, but we must take $\tau_0(\alpha, \beta) = (\alpha s + \beta t) \cdot u_0$ where $V(u_0) = aq$.
Otherwise, the vanishing orders of $\tau_j$ and $u_j$ may be taken as before.

If $n \neq k$, then 
the condition for $\phi$ in \eqref{mat} to be order preserving at $p$ is that $\alpha_{\ell, j} = 0$ for all $\ell, j$ and $c_{\ell, j} = 0$ for all $\ell < j$. The condition for $\phi$ to be order preserving at $q$ is that all $\beta_{\ell, j} = 0$; and $c_{\ell,j} = 0$ for all $\ell > j$ with $(\ell, j) \neq (n, n-1)$; and $c_{n-1,n} = 0$. Note that $c_{n,n-1}$ need not vanish because $\ord_q(\sigma_{n}) > \ord_q(\sigma_{n-1})$. 

In the case $n=k$, the condition for $\phi$ to be order preserving at $q$ is that $c_{\ell,j} = 0$ for all $\ell > j$  and $\beta_{\ell, j} = 0$ for all $(\ell, j) \neq (0, k-1)$. Note that $\beta_{0, k-1}$ is not required to vanish because $\ord_q(\sigma_{k-1}) = a-1 < a = \ord_q(u_0)$ so $\beta_{0, k-1}$ is not required to vanish. Thus, $\dim W_q =\frac{k(k+1)}{2}+1$. Note that $n = k$ corresponds to the case when $L \cong \O((n+m)q)$.
Our explicit description shows that $\dim W_p \cap W_q = k+1$, and the intersection consists of matrices with arbitrary diagonal entries and at most one non-zero off-diagonal entry. Hence, in all cases, $\ker(\diag) \cap W_p \cap W_q$ consists of matrices with at most one non-zero entry.
   \end{proof}

Having characterized necessary compatibility conditions at the nodes and when they are independent,
we now prove \eqref{equiv}. This will be subsumed by the results of the next section, but we include it here as the proof indicates subvarieties of $\Pic^d(\X_0)$ where the limits of line bundles with a certain splitting type must live.

\begin{Lemma} \label{dimbound}
Let $f\colon  C \rightarrow \pp^1$ be a general genus $g$, degree $k$ cover. Then
\[\dim \Sigma_{\vec{e}}(C, f) \leq g - u(\vec{e}).\]
\end{Lemma}
\begin{proof}
The case $g = 1$ was proved in Lemma \ref{ezpush}, so we assume $g > 1$. Since we are also assuming $k > 2$, we can choose a degree distribution $d = d_1 + \ldots + d_g$ so that no $d_i$ is a multiple of $k$. In particular, given $\L_0 = (L_1, \ldots, L_g) \in \Pic^d(\X_0)$, we may assume that $L_i \ncong f^*\O(n)$.
Define
\[\epsilon_i = \begin{cases} 1 & \text{if $i \neq 1, g$ and $L_i \cong  \O_{X_i}(np_{i-1} + mp_i)$} \\ 1 & \text{if $i = 1$ and $L_1 \cong \O_{X_1}(mp_1)$} \\ 1 & \text{if $i = g$ and $L_g \cong \O_{X_g}(np_{g-1})$} \\ 0 &\text{otherwise.} \end{cases}\]
By Lemmas \ref{cond} and \ref{zing},
\begin{align*}
\dim V_{\L} &\leq \sum_{i=1}^g \dim \{\phi \in H^0(P_i, End(E_i)) : \phi \text{ order preserving at nodes on $X_i$}\} - k(g-1) \\
&\leq \frac{k(k+1)}{2} +\epsilon_1 + \frac{k(k+1)}{2} +\epsilon_g + \sum_{i=2}^{g-1} (k + \epsilon_i ) - k(g-1)\\
&\leq k^2 + \delta
\end{align*}
where $\delta$ is the number of $i$ for which $L_i \cong \O_{X_i}(np_{i-1} + mp_{i})$ for some $m, n$ (with $m = 0$ if $i = 1$ and $n = 0$ if $i = g$).  In particular, the codimension of the subvariety of line bundles $\L_0$ in $\Pic^d(X)$ for which $V_{\L} \geq k^2 + \delta$ is at least $\delta$. This implies that for general $\f_t: \X_t \rightarrow \P_t$ in the family $\f\colon  \X \rightarrow \P$, 
\[\dim\{L \in \Pic^d(\X_t) : h^0(\P_t, End((\f_t)_*L)) \geq \delta +k^2\} \leq g - \delta.\]
To finish, note that $h^0(\P_t, End((\f_t)_*L)) \geq k^2+\delta$ implies $u((\f_t)_*L) \geq \delta$, and so for each $\vec{e}$, we have $\dim \Sigma_{\e}(\X_t, f_t) \leq g - u(\vec{e})$ for general $t$. By upper-semicontinuity, this upper bound on $\dim \Sigma_{\vec{e}}(C, f)$ holds for general degree $k$ covers $f\colon  C \rightarrow \pp^1$.
\end{proof}

 \section{Smoothness} \label{smoothness}

In this section, we prove that $\Sigma_{\vec{e}}(C, f)$ is smooth for general $f\colon  C \rightarrow \pp^1$. This should be thought of as an analogue of the Gieseker-Petri theorem, which was first proved by Gieseker \cite{G}, and later by Eisenbud-Harris \cite{EH2} using a degeneration with elliptic tails.

For every $L \in \Pic^d(C)$, there is a natural map
\begin{equation} \label{dmap}
H^1(\pp^1, f_*\O_C) =H^1(C, \O_C) = \Def^1(L) \rightarrow \Def^1(f_*L) = H^1(\pp^1, End(f_*L)),
\end{equation}
sending a first order deformation of $L$ to the induced deformation of the push forward. This map is realized by taking cohomology of the map of sheaves $\eta: f_*\O_C \rightarrow End(f_*L)$ that locally sends a function $z$ on $C$ to the endomorphism ``multiplication by $z$" on $L$, viewed as an $\O_{\pp^1}$ module.
The kernel of \eqref{dmap} is the tangent space to $\Sigma_{\vec{e}}(C, f)$.
Thus, our goal is to show that \eqref{dmap} is surjective for all $L \in \Pic^d(C)$. Indeed, this implies that if $\Sigma_{\vec{e}}(C, f)$ is non-empty, 
\[\dim T_L \Sigma_{\e}(C, f) \leq g - \dim \Def(f_*L) = g - u(\vec{e}) \leq \dim \Sigma_{\vec{e}}(C, f),\]
so it is smooth.

We proceed by showing that the Serre dual of \eqref{dmap},
\begin{equation} \label{sd}
\mu: H^0(\pp^1, End(f_*L) \otimes \omega_{\pp^1}) \rightarrow H^0(\pp^1, (f_*\O_C)^\vee \otimes \omega_{\pp^1}),
\end{equation}
is injective. The kernel of this map is the ``obstruction to smoothness." We think of $H^0(\pp^1, End(f_*L) \otimes \omega_{\pp^1})$ as the subspace of $H^0(\pp^1, End(f_*L))$ vanishing at two prescribed points. The map $\mu$ is thus a restriction of the map on global sections induced by
\[\tilde{\mu}: End(f_*L) \cong End(f_*L)^\vee \to (f_* \O_C)^\vee,\]
which is the composition of the canonical isomorphism $End(f_*L) \cong End(f^*L)^\vee$ with the map dual to $\eta$. For any vector bundle $E$, this isomorphism
$End(E) \cong End(E)^\vee$
 is induced by the perfect pairing $End(E) \times End(E) \to \O$ given by $(\phi, \psi) \mapsto \tr(\phi \cdot \psi)$. Therefore, $\tilde{\mu}$ sends
an endomorphism $\phi \in End(f_*L)(U)$ to the linear functional on $(f_*\O_C)(U)$ given by $z \mapsto \tr(\phi  \cdot \eta(z))$. We will need to know that this map is non-zero on certain elements over components of our degeneration.

\begin{Lemma} \label{nz}
Let $f\colon  X \rightarrow \pp^1$ be a degree $k$ map of an elliptic curve to $\pp^1$ and let $L \in \Pic^d(C)$.
If $\phi \in H^0(\pp^1, End(f_*L))$ is represented by a matrix with a single nonzero entry, then $\tilde{\mu}(\phi) \neq 0$. 
\end{Lemma}
\begin{proof}
For each open subset $U \subset \pp^1$, we have a commutative diagram
\begin{center}
\begin{tikzcd}
& H^0(\pp^1, End(f_*L)) \arrow{d} \arrow{r}{\tilde{\mu}} & \arrow{d} H^0(\pp^1, (f_*\O_C)^\vee) \\
& H^0(U, End(f_*L)) \arrow{r}[swap]{(\tilde{\mu})|_U} & H^0(U, (f_*\O_C)^\vee).
\end{tikzcd}
\end{center}
It suffices to show that the image of $\phi$ in the lower right is nonzero.
Choose $U$ small enough that $L$ is trivialized on $f^{-1}(U)$ and $f_*\O_C$ is trivialized on $U$, so
we have isomorphisms $(f_*L)|_U \cong (f_*\O_C)|_U \cong \O_{U}^{\oplus k}$. By hypothesis, there exists a basis so that, $\phi|_U : \O_U^{\oplus k} \rightarrow \O_U^{\oplus k}$ is represented by a matrix with one non-zero entry, say in the $(i, j)$ slot. These basis vectors of $\O_U^{\oplus k}$ correspond to non-vanishing functions in $\O_C(f^{-1}(U))$. The ratio of the $j$th basis element over the $i$th basis element therefore defines a function $z \in \O_C(f^{-1}(U))$ such that the $(j, i)$ entry of $\eta(z)$ is non-zero.
Hence, $\tr(\phi|_U \cdot \eta(z)) \neq 0$, showing $\tilde{\mu}(\phi|_U)$ is non-zero.
\end{proof}

We now deduce the desired injectivity by studying limits on the central fiber of our degeneration from the previous section, continuing all notation developed there.
Recall that for each $\L_0 \in \Pic^d(\X_0)$, we understood $V_\L \subset H^0(End(\f_*\L_0))$ through its image under the inclusion $\iota$ as compatible tuples $(\phi_1, \ldots, \phi_g)$ in $\bigoplus_{i=1}^g H^0(P_i, End(E_i))$.
\begin{Lemma} \label{sm}
For general $\X_t \rightarrow \P_t $ in our degeneration,
\[
\mu_t: H^0(\P_t, End(f_*\L_t) \otimes \omega_{\P_t}) \rightarrow H^0(\P_t, (f_*\O_{\X_t})^\vee \otimes \omega_{\P_t}),
\]
is injective for all $\L_t \in \Pic^d(\X_t)$. Hence, if it is non-empty, $\Sigma_{\vec{e}}(C, f)$ is smooth for general degree $k$ covers $f\colon  C \rightarrow \pp^1$.
\end{Lemma}
\begin{proof}
Let $\omega$ be the relative dualizing sheaf of $\P \rightarrow \Delta$. 
 Recall that $p_0$ and $p_g$ are points of total ramification of $\f|_{\X_0}$ that are distinct from the nodes. We set $\zeta_1 = f(p_0) \in P_1$ and $\zeta_g = f(p_g) \in P_g$, so we have an isomorphism $\omega|_{\P_0} \cong \O_{\P_0}(-\zeta_1 - \zeta_g)$.

Let $\L$ be given and define $V_\L(-\zeta_1-\zeta_g) \subset V_\L$ to be the subspace of sections vanishing at $\zeta_1$ and $\zeta_g$.
We have a commutative diagram
\begin{center}
\begin{tikzcd}
V_\L(-\zeta_1-\zeta_g) \arrow[swap]{ddr}{\iota} \arrow{r} &H^0(\P_0, End(\f_*\L_0) \otimes \omega|_{\P_0}) \arrow{d} \arrow{r}{\mu} &H^0(\P_0, (\f_*\O_{\X_0})^\vee \otimes \omega|_{\P_0}) \arrow{d} \\
&H^0(\P_0, End(\f_*\L_0)) \arrow{d} \arrow{r}{\tilde{\mu}} &H^0(\P_0, (\f_*\O_{\X_0})^\vee) \arrow{d}\\
 &\bigoplus_{i=1}^g H^0(P_i, End(E_i)) \arrow[swap]{r}{\oplus \tilde{\mu}_i} &\bigoplus_{i=1}^g H^0(P_i, ((f_i)_*\O_{X_i})^\vee).
\end{tikzcd}
\end{center}
By uppersemi-continuity, injectivity of $\mu_t$ for general $t$ follows from showing the composition along the top row is injective. We will show that the composition from the upper left to the lower right along the bottom is injective.

For each $i$, 
let $W_p^i \subset H^0(P_i, End(E_i))$ denote the subspace of endomorphisms on component $i$ that are order preserving at $p$. In addition, let $\diag^i_{p} : W_{p}^i \rightarrow \cc^{\oplus k}$ be defined by $\phi_i \mapsto (d_{p}^{(0)}(\phi_i), \ldots, d_{p}^{(k-1)}(\phi_i))$, which we saw in Lemma \ref{zing} corresponds to taking diagonal entries of a matrix representative for an endomorphism. 
On $W_{p_{i-1}} \cap W_{p_{i}}$, the maps $\diag^i_{p_i}$ and $\diag^i_{p_{i-1}}$ are both defined  and are related by a permutation (they correspond to taking diagonal entries of a matrix in different orders). Hence,
\begin{equation} \label{twoeq}
W_{p_{i-1}} \cap W_{p_{i}} \cap \ker(\diag^i_{p_i}) = W_{p_{i-1}} \cap W_{p_{i}} \cap \ker(\diag^i_{p_{i-1}}).
\end{equation}
Lemma \ref{cond} \eqref{d}, implies that in a compatible tuple, if $\phi_{i-1} \in \ker(\diag^{i-1}_{p_i})$ then $\phi_i \in \ker(\diag^i_{p_i})$.
Note that, taking a matrix representation for $\phi_1$ as in \eqref{mat}, the condition $\phi_1(\zeta_1) = 0$ is that $c_{j, \ell} = 0$ and $\alpha_{\ell, j}  = 0$ for \textit{all} $\ell, j$. Thus, if $\phi_1(\zeta_1) = 0$, then $\phi_1 \in  W_{p_0}^1 \cap \ker(\diag^1_{p_0})$. 
Similarly, if $\phi_g(\zeta_g) = 0$ then $\phi_g \in W_{p_g} \cap \ker(\diag^g_{p_g})$.
By Lemma \ref{cond} and \eqref{twoeq}, we therefore have
\begin{equation*}
\iota(V_\L(-\zeta_1-\zeta_g)) \subseteq \left\{(\phi_1, \ldots, \phi_g) : 
\text{ $\phi_i \in W_{p_{i-1}}^i \cap W_{p_i}^i \cap \ker(\diag^{i}_{p_{i}})$ for $1 \leq i \leq g$}
\right\}.
\end{equation*}
If $g > 1$, then we can choose a degree distribution so that $\deg(L_i)$ is never a multiple of $k$, and hence $L_i$ is never $f^*\O(n)$.
The final sentence of Lemma \ref{zing} then ensures that each $\phi_i$ is represented by a matrix with at most one-nonzero entry. In the case $g = 1$ we need an additional argument if $L_1 = f^*\O(n)$. In this case, choosing any splitting of $f_*L_1$ induces a splitting of $End(f_*L_1)$ where all but one of the matrix entries consist of a constant or linear form, and one entry is quadratic. After imposing vanishing at $\zeta_1$ and $\zeta_g$, only the quadratic entry can be non-zero. In either case, all $\phi_i$ have at most one non-zero entry, so Lemma \ref{nz} now shows that the composition of the inclusion $\iota$ with $\oplus \tilde{\mu}_i$ is injective.
\end{proof}

 \section{Existence} \label{existence}

In this section, we exploit the combinatorial structure of splitting loci stratifications to deduce existence from a simple calculation.
This relies on universal enumerative formulas for splitting loci found in \cite{L}. 

\begin{theorem}[Thm. 1.1 of \cite{L}] \label{mine}
If $E$ is a vector bundle on $\pi: B \times \pp^1 \rightarrow B$ and $\codim \Sigmabar_{\vec{e}}(E) = u(\vec{e})$, then the class $[\Sigmabar_{\vec{e}}(E)]$ is given by a universal formula in terms of Chern classes $\pi_* E(m)$ and $\pi_*E(m-1)$ for suitably large $m$. Moreover, if this expected class is non-zero, then $\Sigmabar_{\vec{e}}(E)$ is non-empty.
\end{theorem}

To make use of the above theorem, we need the Chern classes of push forwards of twists of the vector bundle $\E$ defined in Section \ref{split}.

\begin{Lemma} \label{cE}
Let $f\colon  C \rightarrow \pp^1$ be a curve of genus $g$ with a degree $k$ map to $\pp^1$. 
Let $\mathcal{L}$ be a Poincar\'e line bundle on $C \times \Pic^d(C)$ and let $\E = (f \times \mathrm{id})_* \mathcal{L}$ on $\pp^1 \times \Pic^d(C)$. Let $\pi: \pp^1 \times \Pic^d(C) \rightarrow \Pic^d(C)$ be the projection.
Let $\theta$ denote the class of the theta divisor on $\Pic^d(C) = \Jac(C)$.
Then we have $c_i(\pi_*\E(m)) = (-1)^i\theta^i/i!$ modulo classes supported on $\Supp(R^1\pi_*\E(m))$. The total Chern class is $c(\pi_*\E(m)) = e^{-\theta}$ away from $\Supp(R^1\pi_*\E(m))$.
\end{Lemma}
\begin{proof}
We have a commutative diagram
\begin{center}
\begin{tikzcd}
& C \arrow{r}{f} & \pp^1 \\
&C \times \Pic^d(C) \arrow{u}{\alpha} \arrow{r}{f \times \mathrm{id}} \arrow{rd}[swap]{\nu} & \arrow{u}[swap]{\beta} \pp^1 \times \Pic^d(C) \arrow{d}{\pi}\\
& & \Pic^d(C)
\end{tikzcd}
\end{center}
Let $H = f^*(\O_{\pp^1}(1))$.
By the projection formula, 
\[\E(m) = ((f\times \mathrm{id})_* \L) \otimes \beta^* \O_{\pp^1}(m) =(f \times \mathrm{id})_*(\L \otimes \alpha^* H^{\otimes m}),\]
and so $\pi_*\E(m) = \nu_*(\L \otimes \alpha^* H^{\otimes m})$. We have that $\L \otimes \alpha^*H^{\otimes m}$ is the pullback of a Poincar\'e line bundle on $C \times \Pic^{d+mk}(C)$ via the identification $\Pic^d(C) \rightarrow \Pic^{d+mk}(C)$ given by tensoring with $H^{\otimes m}$. The calculation in \cite[p. 336]{ACGH} determines the Chern classes of the push forward of a Poincar\'e line bundle $\nu_*(\L \otimes \alpha^*H^{\otimes m})$ away from $\Supp(R^1\nu_*(\L \otimes \alpha^*H^{\otimes m})) = \Supp(R^1 \pi_*\E(m))$.
\end{proof}

Given the Chern classes of $\pi_*\E(m)$, the classes of splitting degeneracy loci are (in theory) computable by the techniques of \cite{L}.
\begin{Example} \label{revisit}
Continuing Example \ref{Geoff}, the classes of the Brill-Noether splitting degeneracy loci on $\Pic^4(C)$ for $C$ a general trigonal curve of genus $5$ are
\[[\Sigmabar_{(-2, -1, 0)}(\E)] = \theta, \qquad [\Sigmabar_{(-2, -2, 1)}(\E)] = [\Sigmabar_{(-3, 0, 0)}(\E)] = \frac{\theta^4}{24}, \qquad [\Sigmabar_{(-3, -1, 1)}(\E)] = \frac{\theta^5}{60}.\]
The first three classes are computed using \cite[Lemma 5.1]{L}. The last class comes from twisting and substituting the Chern classes from Lemma \ref{cE} into the universal formula found in \cite[Example 6.2]{L}. Notice that $[\Sigmabar_{(-3, -1, 1)}(\E)]$ is twice the class of a point, as found in Example \ref{Geoff}.
Also, $[W^1_4(C)] = [\Sigmabar_{(-2, -2, 1)}(\E)] + [\Sigmabar_{(-3, 0, 0)}(\E)] = \frac{\theta^4}{12}$ is the class computed by Kempf-Kleimann-Laksov \cite{K, KL1, KL2}.
\end{Example}

The universal formulas guaranteed by Theorem \ref{mine} are difficult to compute in general, but Lemma \ref{cE} implies the following remarkable fact.
 Given a splitting type $\vec{e}$, let $|\vec{e}| = e_1 + \ldots + e_k$.

\begin{Lemma} \label{hoho}
Fix $k$ and $\vec{e} = (e_1, \ldots, e_k)$. Given $f\colon  C \rightarrow \pp^1$ a genus $g$ curve with degree $k$ map to $\pp^1$, let $d = g+k+ |\vec{e}| - 1$.
The expected class of $\overline{\Sigma}_{\vec{e}}(C,f)$ in $\Pic^d(C)$ is $a_{\vec{e}} \cdot \theta^{u(\vec{e})}$ for some constant $a_{\vec{e}} \in \qq$ depending only on $\vec{e}$ (independent of $g$).
\end{Lemma}
\begin{proof}
The loci $\overline{\Sigma}_{\vec{e}}(C,f)$ are splitting loci of the rank $k$, degree $|\vec{e}|$ vector bundle $\E = (f \times \mathrm{id})_* \L$ on $\pp^1 \times \Pic^d(C)$.
By Theorem \ref{mine}, the expected class of $\overline{\Sigma}_{\vec{e}}(C,f)$ is given by a universal formula, depending only on $\vec{e}$, in terms of the Chern classes of $\pi_* \E(m)$ for suitably large $m$. The $i$th Chern class of this vector bundle is a multiple of $\theta^i$ that does not depend on $g$ by Lemma \ref{cE}.
\end{proof}

\begin{remark}
For a fixed $k$, a choice of $|\vec{e}|$ determines an allowed difference $d - g = k+|\vec{e}| - 1$. The above is therefore akin to the observation that the formula for the class of $W^r_d(C)$ for general $C$ in $\M_g$ computed by Kempf-Kleiman-Laksov \cite{K, KL1, KL2} depends only on $d - g$.
\end{remark}

Lemma \ref{hoho} allows us to leverage the combinatorics of the partial ordering to deduce existence from calculations for certain special splitting types. Following \cite{L}, let us write $(-n, *, \ldots, *)$ to denote the splitting type of $\O_{\pp^1}(-n) \oplus B(k-1, |\vec{e}|+n)$.

\begin{Lemma} \label{comp}
For every $\e$, there exists $n$ such that $(-n, *, \ldots, *)\leq \e$.
We have $a_{(-n,*,\ldots, *)} = 1/u(-n,*,\ldots,*)!$.
\end{Lemma}
\begin{proof}
We may take $n = -(|\vec{e}| + e_1k)$.
Notice that $\Supp(R^1\pi_*\E(n-1)) = \Sigmabar_{(-n-1, *, \ldots, *)}$, which has codimension larger than $u(-n, *, \ldots, *)$. Therefore, we may calculate the class of $\Sigmabar_{(-n,*, \ldots, *)}$ on the complement of $\Supp(R^1\pi_*\E(n-1))$. On the complement, Lemma \ref{cE} says that $c((\pi_*\E(n-1))^\vee) = c((\pi_*\E(n))^\vee)  = e^{\theta}$.
By \cite[Lemma 5.1]{L},
\[[\Sigmabar_{(-n,*, \ldots, *)}] = \left[\frac{c((\pi_*\E(n-1))^\vee)^2}{c((\pi_* \E(n))^\vee)} \right]_{u(-n,*,\ldots, *)} = \left[\frac{(e^{\theta})^2}{e^\theta}\right]_{u(-n,*,\ldots, *)}=  \frac{\theta^{u(-n,*,\ldots,*)}}{u(-n,*,\ldots,*)!}\]
as desired.
\end{proof}

\begin{proof}[Proof of Theorem \ref{main}]
We will show that $a_{\vec{e}}$ is non-zero for all $\e$. By the second half of Theorem \ref{mine}, this will imply $\Sigmabar_{\vec{e}}(C, f)$ is non-empty whenever $u(\vec{e}) \leq g$. Then, Lemmas \ref{low} and \ref{dimbound} show that $\Sigmabar_{\vec{e}}(C, f)$ has dimension $g - u(\vec{e})$ and is the closure of $\Sigma_{\vec{e}}(C, f)$. Lemma \ref{sm} shows $\Sigma_{\vec{e}}(C, f)$ is smooth.

Fix $\vec{e}$ and choose $n$ such that $(-n, *, \ldots, *) \leq \e$. Choose any $g' \geq u(-n, *, \ldots, *)$ and let $f': C' \rightarrow \pp^1$ be a general point of $\H_{k,g'}$. Let $d' = g' + k + |\vec{e}| - 1$. By Lemma \ref{comp}, $\overline{\Sigma}_{(-n, *, \ldots, *)}(C',f') \subset \Pic^{d'}(C')$ is non-empty. Thus, $\overline{\Sigma}_{\e}(C', f') \subset \Pic^{d'}(C')$ is non-empty too. By Lemmas \ref{low} and \ref{dimbound}, $\codim \overline{\Sigma}_{\e}(C', f') = u(\e)$. Being non-empty of the expected codimension on a projective variety, $[\overline{\Sigma}_{\e}(C', f') ] = a_{\vec{e}} \theta^{u(\vec{e})} \neq 0$ on $\Pic^{d'}(C')$. Hence $a_{\vec{e}} \neq 0$, as desired.
\end{proof}


\begin{thebibliography}{}

\bibitem{ACGH}
E.~Arbarello, M.~Cornalba, P.~Griffiths, and J.~Harris, \textit{Geometry of Algebraic Curves}, Springer (1985).

\bibitem{CPJ} K.~Cook-Powell and D.~Jensen, \textit{Irreducible Components of Brill-Noether Loci of General Curves with Fixed Gonality}, preprint (2019).

\bibitem{CM1}
M. Coppens and G. Martens, \textit{Linear series on a general $k$-gonal curve}, Abhandlungen aus dem Mathematischen Seminar der Universitt Hamburg, 69:347, 1999.

\bibitem{CM2}
M.~Coppens and G.~Martens, \textit{On the varieties of special divisors}, Koninklijke Nederlandse Akademie van Wetenschappen, Indagationes Mathematicae, New Series, 13(1):29, 2002.

\bibitem{EH1} D.~Eisenbud and J.~Harris, \textit{Limit linear series: Basic theory}, Inventiones Math. \textbf{85} (1986), 337-371.

\bibitem{EH2} D.~Eisenbud and J.~Harris, \textit{A simpler proof of the Geiseker-Petri Theorem on special divisors}, Inventiones Math. \textbf{74} (1983), 269-280.


\bibitem{EH} D. Eisenbud and J. Harris, \textit{3264 \& All That}, Cambridge University Press, 2016.

\bibitem{ES} D.~Eisenbud, F.-O.~Shreyer, \textit{Relative Beilinson monad and direct image for families of coherent sheaves}, Trans. Amer. Math. Soc. \textbf{360}:10 (2008), 5367--5396.

\bibitem{FL}
W.~Fulton and R.~Lazarsfeld, \textit{On the connectedness of degeneracy loci and special divisors}, Acta Math. Vol. 146 (1981), 271--283.

\bibitem{G} D.~Gieseker, \textit{Stable curves and special divisors: Petri's conjecture}, Invent. Math. \textbf{66} (1982), 251--275.

\bibitem{GH} P.~Griffiths and J.~Harris, \textit{On the variety of special linear systems on a general algebraic curve}, Duke Math. J, 47 (1980), 233--272.

\bibitem{HM} J.~Harris and I.~Morrison, \textit{Moduli of Curves}, Springer, 1998.

\bibitem{admiss} J.~Harris and D.~Mumford, \textit{On the Kodaira dimension of the moduli space of curves}, Invent. Math. \textbf{67} (1982),
no. 1, 23?88, With an appendix by William Fulton.

\bibitem{JR}
D.~Jensen and D.~Ranganathan, \textit{Brill-Noether theory for curves of a fixed gonality}, arXiv:1701.06579, (2017).

\bibitem{K} 
G.~Kempf, \textit{Schubert methods with an application to algebraic curves}, Publ. Math. Centrum, Amsterdam (1971).

\bibitem{KL1} 
S.~Kleiman and D.~Laksov, \textit{On the existence of special divisors}, Amer. J. Math., 94 (1972), 431--436.

\bibitem{KL2}
S.~Kleiman and D.~Laksov, \textit{Another proof of the existence of special divisors}, Acta Math. 132 (1974), 163--176.

\bibitem{L}
H.~Larson, \textit{Universal degeneracy classes for vector bundles on $\pp^1$ bundles}, arXiv:1906.10290, (2019).

\bibitem{liu} 
Q.~Liu, \textit{Reduction and lifting of finite covers of curves}, Proceedings of the 2003 Workshop on Cryptography and Related Mathematics, Chuo University, 2003, 161?180.

\bibitem{Pf}
N.~Pflueger, \textit{Brill-Noether varieties of k-gonal curves}, arXiv:1603.08856, (2016).

\end{thebibliography}
\end{document}